\documentclass [12pt,a4paper]{article}
\usepackage{amsfonts,amssymb,amsmath,amscd,latexsym,makeidx,theorem,color}
\usepackage[latin1]{inputenc}

\author{Gabriele Mancini\thanks{The authors are supported by Swiss National Science Foundation, project nr. PP00P2-144669. \hspace{1cm} Date: August 25, 2016. Revised: May 5, 2017.} \\ \small Universit\"at Basel \\ \footnotesize \texttt{gabriele.mancini@unibas.ch} \and  Luca Martinazzi$^*$ \\ \small Universit\"at Basel  \\ \footnotesize \texttt{luca.martinazzi@unibas.ch}}
\title{The Moser-Trudinger inequality and its extremals on a disk via energy estimates}
\date{}
\newtheorem{trm}{Theorem}
\newtheorem{prop}[trm]{Proposition}
\newtheorem{cor}[trm]{Corollary}
\newtheorem{lemma}[trm]{Lemma}

\newtheorem{OP}{Open problem}
\newtheorem{thmx}{Theorem}
\newcommand{\R}{\mathbb{R}}

\newcommand{\de}{\partial}

\newcommand{\M}[1]{\mathcal{#1}}

\newcommand{\bra}[1]{\left({#1}\right)}

\newcommand{\ph}{\varphi}
\newcommand{\ra}{\rightarrow}

\newenvironment{proof}{\noindent\emph{Proof.}}{\phantom{ } \hfill$\square$\medskip}

\newenvironment{rmk}{\medskip\noindent\emph{Remark.}}{\medskip}

\DeclareMathOperator{\loc}{loc}

\newcommand{\be}{\begin{equation*}}
\newcommand{\ee}{\end{equation*}}

\newcommand{\bdm}{\begin{displaymath}}
\newcommand{\edm}{\end{displaymath}}

\begin{document}
\maketitle

%\begin{abstract} We study the Dirichlet energy of non-negative radially symmetric critical points $u_\mu$ of the Moser-Trudinger inequality on the unit disc in the plane, and prove that it expands as $$4\pi+\frac{4\pi}{\mu^{4}}+o(\mu^{-4})\le \int_{B_1}|\nabla u_\mu|^2dx\le 4\pi+\frac{6\pi}{\mu^{4}}+o(\mu^{-4}),\quad \text{as }\mu\to\infty.$$
%where $\mu=u_\mu(0)$ is the maximum of $u_\mu$. As a consequence we obtain a new proof of the Moser-Trudinger inequality, of the Carleson-Chang result about the existence of extremals for the same inequality, and of the Struwe and Lamm-Robert-Struwe result about the existence of at least two critical points in the supercritical regime (only in the case of the unit disk).
%
%Our results are stable under perturbations of the Moser-Trudinger inequality sufficiently weak, and we identify the critical level of perturbation for which, although the perturbed Moser-Trudinger inequality still holds, the energy of its critical points converges to $4\pi$ \emph{from below}, and we expect in some of these cases, that the existence of extremals does not hold, nor the existence of critical points in the supercritical regime.
%\end{abstract}

\begin{abstract} We study the Dirichlet energy of non-negative radially symmetric critical points $u_\mu$ of the Moser-Trudinger inequality on the unit disc in $\R^2$, and prove that it expands as $$4\pi+\frac{4\pi}{\mu^{4}}+o(\mu^{-4})\le \int_{B_1}|\nabla u_\mu|^2dx\le 4\pi+\frac{6\pi}{\mu^{4}}+o(\mu^{-4}),\quad \text{as }\mu\to\infty,$$
where $\mu=u_\mu(0)$ is the maximum of $u_\mu$. As a consequence, we obtain a new proof of the Moser-Trudinger inequality, of the Carleson-Chang result about the existence of extremals, and of the Struwe and Lamm-Robert-Struwe multiplicity result in the supercritical regime (only in the case of the unit disk).\\
Our results are stable under sufficiently weak perturbations of the Moser-Trudinger functional. We explicitly identify the critical level of perturbation for which, although the perturbed Moser-Trudinger inequality still holds, the energy of its critical points converges to $4\pi$ \emph{from below}. We expect, in some of these cases, that the existence of extremals does not hold, nor the existence of critical points in the supercritical regime.
%{\bf AMS subject classification: 35B33,  35B44, 35B38, 35J61}
\end{abstract}

%\noindent{\it Key Words:} Moser-Trudinger inequality, critical points, blow-up analysis.

%\bigskip

\section{Introduction}

Consider the Moser-Trudinger inequality in dimension two (see \cite{mos,poh,tru}):
\begin{thmx}[Moser \cite{mos}]\label{trmMoser}
For $\Omega\subset\R^2$ with finite measure $|\Omega|$ we have
\begin{equation}\label{MT}
\sup_{u\in H^1_0(\Omega):\|\nabla u\|_{L^2}^2\le 4\pi}\int_{\Omega}e^{u^2}dx\le C|\Omega|.
\end{equation}
Moreover the constant $4\pi$ is sharp.
\end{thmx}
As noticed by Moser, the subcritical inequality
\begin{equation}\label{MTalpha}\tag{$I_\alpha$}
\sup_{u\in H^1_0(\Omega):\|\nabla u\|_{L^2}^2\le \alpha}\int_{\Omega}e^{u^2}dx\le \frac{|\Omega|}{1-\frac{\alpha}{4\pi}},
\end{equation}
 is easy to obtain for $\alpha < 4\pi$. Indeed, by symmetrization and scaling one reduces to the case of the unit disk $\Omega=B_1$ and $u=u(r)$ radially symmetric. Then, by the fundamental theorem of calculus and H\"older's inequality, one bounds
\begin{equation}\label{stimaholder}
|u(r)|^2\le \left(\int_r^1 |u'(\rho)|d\rho \right)^2 \le \int_r^1 2\pi \rho |u'(\rho)|^2d\rho \int_r^1 \frac{d\rho}{2\pi \rho}\le \frac{\|\nabla u\|_{L^2}^2}{2\pi}\log\frac{1}{r},
\end{equation}
hence if $\|\nabla u\|_{L^2}^2\le \alpha<4\pi$,
$$\int_{B_1} e^{u^2}dx\le \int_0^1 2\pi r e^{\frac{\alpha}{2\pi}\log\frac{1}{r}} dr=2\pi\int_0^1 r^{1-\frac{\alpha}{2\pi}}dr=\frac{\pi}{1-\frac{\alpha}{4\pi}}.$$
The difficult part of Theorem \ref{trmMoser} is to prove that \eqref{MT} also holds with the critical constant $4\pi$. To do that Moser considers a special class of functions, which are now known as Moser-functions or broken-line functions, and notices that for such functions \eqref{MT} holds (and it fails if we replace $4\pi$ by a larger constant). Further he shows that any function for which \eqref{stimaholder} is close to an identity at one point must be close to a Moser function in a suitable sense.
%again satisfying \eqref{MT}. 

\medskip

The existence of maximizers  (usually called \emph{extremals})  for the Moser-Trudinger inequality has been pioneered for $\Omega=B_1$ by L. Carleson and A. Chang \cite{CC}: 
\begin{thmx}[Carleson-Chang \cite{CC}]\label{trmCC} When $\Omega=B_1$ is the unit disk, the inequality \eqref{MT} admits an extremal.
\end{thmx}

The original proof of Theorem \ref{trmCC} is based on estimating
$$F(u):=\int_{B_1}e^{u^2}dx$$
on a sequence $u_k$ maximizing the supremum in \eqref{MT}, and showing, in a very clever way, that $\limsup_{k\to\infty} F(u_k)\le \pi(1+ e)$ if the sequence blows-up. On the other hand, this cannot be the case, since the authors exhibit a function $u^*$ such that $F(u^*)>\pi (1+e)$. Then the sequence $(u_k)$ is precompact and converges to a maximizer. This method has been extended to several more general cases, starting from the works of Struwe \cite{stru}, Flucher \cite{flu} and Li \cite{Li}.

\medskip

In this paper we shall give an alternative approach to Theorems \ref{trmMoser} and \ref{trmCC}, based on estimating the Dirichlet energy of the extremals of subcritical inequalities. Indeed it is easy to prove that the subcritical inequality \eqref{MTalpha} has a maximizer $u_\alpha$ for every $\alpha<4\pi$, see Proposition \ref{propsubcrit2} below.
Such extremal satisfies
\begin{equation}\label{eqMTalpha}
-\Delta u_\alpha =\lambda_\alpha u_\alpha e^{u_\alpha^2},
\end{equation}
for a positive Lagrange multiplier $\lambda_\alpha$. The crucial question is whether $u_\alpha$ converges as $\alpha\uparrow 4\pi$. The answer is affirmative and follows easily from the energy estimate of the next theorem, which is the core of our argument. %Then in Section \ref{sec2} we will show how to use this to prove Theorems \ref{trmMoser} and \ref{trmCC}.

\begin{trm}\label{trmEnergy}
Let $(u_k)\subset H^1_0(B_1)$ be any sequence (possibly unbounded) of radially symmetric and positive solutions\footnote{Actually the radial symmetry follows from positivity and the moving plane technique.} to 
\begin{equation}\label{eqMTk}
-\Delta u_k =\lambda_k u_k e^{u_k^2},
\end{equation}
for some $\lambda_k>0$. Assume
\begin{equation}\label{defmuk}
\mu_k:=u_k(0)=\max_{B_1}u_k\to \infty,\quad \text{as }k\to\infty.
\end{equation}
Then
\begin{equation}\label{eqEnergy}
4\pi +\frac{4\pi}{\mu_k^4}+o(\mu_k^{-4})\le \|\nabla u_k\|_{L^2}^2\le  4\pi +\frac{6\pi}{\mu_k^4}+o(\mu_k^{-4}).
\end{equation}
\end{trm}

To prove Theorem \ref{trmEnergy} we build up on a technique introduced in \cite{MM} and perform a Taylor expansion of the solutions $u_k$ near the origin, which needs to be precise enough to obtain \eqref{eqEnergy}, see Section \ref{sec3}.

%The proof of Theorem \ref{trmEnergy} relies on a Taylor expansion of the solutions $u_k$ near the origin, which needs to be precise enough to obtain \eqref{eqEnergy}, see Section \ref{sec3}.

\medskip

%Given Theorem \ref{trmEnergy}, for every $k$, let $u_k$ be an extremal of \eqref{MTalpha} for $\alpha=4\pi-\frac{1}{k}$.
%as will be given by Proposition \ref{propsubcrit2}. 
%By Theorem \ref{trmEnergy} the values $u_k(0)=\max_{B_1}u_k$ must remain bounded, and by elliptic estimates, $u_k\to u_\infty$, which will be an extremal of \eqref{MT}, hence proving Theorem \ref{trmCC}. This also implies that \eqref{MT} holds for the critical value $4\pi$. The details will be provided in the next section.

Consider now a mildly perturbed, though completely equivalent version of Theorem \ref{trmMoser}, namely for $\alpha\in (0,4\pi]$ replace \eqref{MT} and \eqref{MTalpha} with
\begin{equation}\label{MTpert}\tag{$I_{\alpha}^g$}
\sup_{u\in H^1_0(\Omega):\|\nabla u\|_{L^2}^2\le \alpha}\int_{\Omega}(1+g(u))e^{u^2}dx\le C_{g,\alpha},
\end{equation}
where
\begin{equation}\label{condg}
g\in C^1(\R),\quad \inf_{\R} g>-1,\quad g(t)=g(-t) \quad \text{and} \quad  \lim_{|t|\to\infty}g(t)=0.
\end{equation}
We want to investigate whether an analog of Theorem \ref{trmEnergy} holds for positive critical points of \eqref{MTpert}, and consequently whether $(I^g_{4\pi})$ admits an extremal. As we shall now see, this is the case if $g$ decays well enough at infinity. More precisely, observe that the critical points of \eqref{MTpert} satisfy
\begin{equation}\label{eqpert0}
-\Delta u=\lambda \left(1+g(u)+\frac{g'(u)}{2u}\right)u e^{u^2}=\lambda(1+h(u))ue^{u^2},
\end{equation}
for some $\lambda\in \R$, where we set
\begin{equation}\label{defh}
h(t):=g(t)+\frac{g'(t)}{2t}, \quad t\in \R\setminus\{0\}.
\end{equation}
%The behaviour of $g(t)$ at infinity will be determinant, in the sense that we look for an asymptotic behaviour of $g(t)$ as $t\to+\infty$ which would guarantee that the analogs of Theorems \ref{trmEnergy}, and therefore of Theorem \ref{trmCC} still hold for \eqref{MTpert}. In this direction we shall assume
We further assume
\begin{equation}\label{condh1}
\inf_{(0,\infty)}h>-1, \quad \sup_{(0,\infty)} h<\infty,\quad \lim_{t\to \infty}t^2h(t)=0
\end{equation}
and
\begin{equation}\label{condh2}
\lim_{t\to \infty}\sup_{|s|\le 1}t^4\left| h\left(t+\frac{s(8\log t +1)}{t}\right)-h(t)\right|=0.
\end{equation}
A typical function $g$ that we have in mind is $g(t)=|t|^{-p}$ near infinity for some $p>2$. More generally one can take a function $\chi\in C^\infty([0,\infty))$ with $\chi\equiv 0$ on $[0,1]$, $\chi\equiv 1$ on $[2,\infty)$, and consider for $R>0$ sufficiently large
\begin{equation}\label{exh1}
g(t)=a\chi(R^{-1}|t|)\log^q(|t|)|t|^{-p},\quad a,q\in \R,\, p>2,
\end{equation}
or even the oscillating function
\begin{equation}\label{exh2}
g(t)=a\chi(R^{-1}|t|)\cos(\log|t|)|t|^{-p},\quad a\in \R,\,p>2.
\end{equation}

Then we have the following generalized versions of Theorems \ref{trmEnergy} and \ref{trmCC}.

\begin{trm}\label{trmEnergypert}
Let $(u_k)\subset H^1_0(B_1)$ be a sequence of radially symmetric and positive solutions to 
\begin{equation}\label{eqMTpert}
-\Delta u_k =\lambda_k (1+h(u_k))u_k e^{u_k^2},
\end{equation}
with $\lambda_k>0$ and $h:(0,\infty)\to\R$ satisfying \eqref{condh1}-\eqref{condh2}. Assume that \eqref{defmuk} holds. Then
\begin{equation}\label{eqEnergypert}
4\pi +\frac{4\pi}{\mu_k^4}+o(\mu_k^{-4})\le \|\nabla u_k\|_{L^2}^2\le  4\pi +\frac{4\pi+ 2\pi(1+\sup h) }{\mu_k^4}+o(\mu_k^{-4}).
\end{equation}
\end{trm}

%Theorem \ref{trmEnergypert} gives us the following analog of Theorem \ref{trmCC}, see Section \ref{sec2}.

\begin{cor}\label{corCC}
If $g$ satisfies \eqref{condg} and $h$ as in \eqref{defh} satisfies \eqref{condh1} and \eqref{condh2}, then $(I_{4\pi}^g)$ with $\Omega= B_1$ admits an extremal.
\end{cor}

It is natural to ask how sharp conditions \eqref{condh1} and \eqref{condh2} are. The following example shows that the quadratic decay is indeed critical.

\begin{trm}\label{trmexample} Let $h:\R\to [-1/2,1/2]$ satisfy $h(t)=-at^{-2}$ for $t\ge R$ for some $a>0$ and $R>0$ fixed, and let $(u_k)\subset H^1_0(B_1)$ be a sequence of radially symmetric positive solutions to \eqref{eqMTpert} satisfying \eqref{defmuk}. Then
$$4\pi +\frac{4\pi-4\pi a}{\mu_k^4}+o(\mu_k^{-4})\le \|\nabla u_k\|_{L^2}^2\le  4\pi +\frac{4\pi +2\pi(1+\sup h)-4\pi a}{\mu_k^4}+o(\mu_k^{-4}).$$
In particular for $a>\frac 32+\frac{\sup h}{2}$ we can find a value $\bar \mu$ such that for any positive solution $u$ to \eqref{eqpert0} with $u(0)\ge \bar \mu$ we have $\|\nabla u\|_{L^2}^2<4\pi$.
\end{trm}

\begin{OP}\label{OP1} Can one find a function $h$ as in Theorem \ref{trmexample} and satisfying \eqref{defh} for some $g$ as in \eqref{condg} such that $\|\nabla u\|_{L^2}^2<4\pi$ for \emph{every} positive $u$ solving \eqref{eqpert0}?
\end{OP}
The function $h$ given in Theorem \ref{trmexample} (and a corresponding function $g$ can be easily constructed) covers the case when $u(0)$ is sufficiently large but one should also rule out the possibility that some ``small'' solutions have energy at least $4\pi$. If the above question has  a positive answer, for such functions $g$ and $h$ one would have that $(I_{4\pi}^g)$ admits no extremal. %, since such extremal (call it $u$) would be positive (possibly up to a reflection), satisfy \eqref{eqpert0} and have energy $\|\nabla u\|_{L^2}^2=4\pi$.
The non-existence of extremals for a very mildly perturbed Moser-Trudinger inequality originally motivated our interest in Theorems \ref{trmEnergypert} and \ref{trmexample}. In \cite{pru} Pruss showed the existence of a function $g$ as in \eqref{condg} such that the inequality $(I_{4\pi}^g)$ does not have extremals. However his construction of $g$ is quite implicit and we do not know its asymptotic behaviour at infinity. More generally the following appears to be open:

\begin{OP} For which functions $g$ as in \eqref{condg} does the perturbed Moser-Trudinger inequality $(I_{4\pi}^g)$ have an extremal?
\end{OP}

Finally the following result will easily follow from Theorem \ref{trmEnergypert}.

\begin{trm}\label{punticritici} Set
$$E(u):=\int_{B_1} (1+g(u))e^{u^2}dx,\quad M_\Lambda:=\left\{u\in H^1_0(B_1): \|\nabla u\|_{L^2}^2=\Lambda\right\},$$
where $g\in C^2(\R)$ satisfies \eqref{condg}, \eqref{defh}, \eqref{condh1} and \eqref{condh2}. % is as in Theorem \ref{trmEnergypert}. 
Then there exists $\Lambda^*>4\pi$ such that for every $\Lambda\in (4\pi,\Lambda^*)$ the functional $E|_{M_\Lambda}$ has at least two positive critical points.
\end{trm}

Theorem \ref{punticritici} for a general smoothly bounded domain $\Omega\subset\R^2$ and with $g\equiv 0$ was proven in \cite{LRS,stru} using variational methods, geometric flows, a sharp quantization estimate, and a monotonicity technique, see also \cite{DMR}. 

\medskip

The paper is organized as follows. In Section \ref{sec2} we will show how the energy estimates of Theorems \ref{trmEnergy} and \ref{trmEnergypert} imply Theorems \ref{trmMoser},  \ref{trmCC} and Corollary \ref{corCC}, while the proofs of Theorems \ref{trmEnergy}, \ref{trmEnergypert}, \ref{trmexample} and \ref{punticritici} are contained in Sections \ref{sec3}, \ref{sec4}, \ref{sec5} and \ref{sec6} respectively. Finally, in the last section, we collect some open problems.
While attempting to avoid repetitions, %in the proofs of Theorems \ref{trmEnergy}, \ref{trmEnergypert} and \ref{trmexample}, 
we had to allow some redundancy to keep the paper reader-friendly. The proof of Theorem \ref{trmEnergy} is the most detailed, and some parts of it will be reused when proving Theorems \ref{trmEnergypert} and \ref{trmexample}.

\paragraph{Notations} For a non-vanishing function $f:(0,\infty)\to \R$ we use the Peano notation $o(f(t))$ and $O(f(t))$ to denote functions such that $o(f(t))/f(t)\to 0$ and $|O(f(t))/f(t)|\le C$ as $|t|\to\infty$.

Since all function we use are radially symmetric, we will use the notation $u(x)=u(r)$ with $x\in\R^2$, $r=|x|$, and also write $\Delta u(r)=r^{-1}(ru'(r))'$.

\paragraph{Acknowledgements} The authors are grateful to the referee for some useful remarks that improved the presentation of the paper.

\section{Proof Theorems \ref{trmMoser} and \ref{trmCC} using Theorem \ref{trmEnergy}}\label{sec2}

In this section we prove Theorems \ref{trmMoser} and \ref{trmCC} starting from the subcritical inequality \eqref{MTalpha} and the energy estimate in Theorem \ref{trmEnergy}. In fact we will be more general and work directly with \eqref{MTpert}, showing that Corollary \ref{corCC} follows from Theorem \ref{trmEnergypert}.

\begin{prop}\label{propsubcrit2}
Assume that $g$ and $h$ satisfy \eqref{condg}, \eqref{defh} and the first condition in \eqref{condh1}. Then for any $\alpha <4\pi$  the inequality \eqref{MTpert} has an extremal $u_\alpha>0$ satisfying \eqref{eqpert0} for some $\lambda\in \bra{0,\frac{\lambda_1(\Omega)}{1+\inf h}}$. Here $\lambda_1(\Omega)$ is the first eigenvalue of $-\Delta$ on $\Omega$ with Dirichlet boundary condition. If $\Omega=B_1$, then $u_\alpha$ can be taken radially symmetric and decreasing.
\end{prop}

\begin{proof} Let $(u_k)\subset H^1_0(\Omega)$ with $\|\nabla u_k\|_{L^2}^2\le \alpha$ be a maximizing sequence for \eqref{MTpert}. By the compactness of the embedding of $H^1_0(\Omega)$ into $L^2(\Omega)$ we have that, up to a subsequence, $u_k\to u_\alpha$ weakly in $H^1_0(\Omega)$, strongly in $L^2(\Omega)$ and almost everywhere.

Fix now $\alpha' \in (\alpha, 4\pi)$.  Since $\tilde u_k:=\sqrt{\frac{\alpha'}{\alpha}}u_k$ satisfies $\|\nabla \tilde u_k\|_{L^2}^2\le \alpha'$, using $(I_{\alpha'})$, we have for any $L>0$
\[
\int_{\{x\in \Omega: u_k\ge L\}} (1+g(u_k))e^{ u_k^2}dx\le e^{-\bra{\frac{\alpha'}{\alpha}-1}L^2}(1+\sup_{\R}g)\int_\Omega e^{\tilde u_k^2}dx =O\bra{e^{-\bra{\frac{\alpha'}{\alpha}-1}L^2} }=o(1)
\]
as $L\to\infty$, uniformly in k. Then, by Lemma \ref{lemmafkL} we infer that
\[\begin{split}
\int_\Omega (1+g(u_\alpha)) e^{u_\alpha^2 }dx&=\lim_{k\to\infty} \int_\Omega (1+g(u_k))e^{ u_k^2}dx \\
&=\sup_{u\in H^1_0(\Omega),\,\|\nabla u\|_{L^2}^2\le \alpha}\int_{\Omega}(1+g(u))e^{u^2}dx.
\end{split}\]
Since
$$\|\nabla u_\alpha\|_{L^2}^2\le \limsup_{k\to\infty}\|\nabla u_k\|_{L^2}^2\le \alpha,$$ we have that indeed $u_\alpha$ is an extremal for \eqref{MTpert}. Since \eqref{defh}-\eqref{condh1} imply that $(1+g(t))e^{t^2}$ is increasing for $t\ge 0$, we have that $\| \nabla u_\alpha\|^2_{L^2}=\alpha$. In particular $u_\alpha$ solves the Euler-Lagrange equation \eqref{eqpert0} for some $\lambda \in \R$.  
Multiplying \eqref{eqpert0} by $u_\alpha$ and integrating we obtain
$$\int_\Omega |\nabla u_\alpha|^2dx=\lambda\int_\Omega (1+h(u_\alpha)) u_\alpha^2 e^{u_\alpha^2}dx>\lambda(1+\inf h) \int_\Omega u_\alpha^2dx,$$
and using the variational characterization of $\lambda_1(\Omega)$ we infer $\lambda\in \bra{0,\frac{\lambda_1(\Omega)}{1+\inf h}}$.

That $u_\alpha$ has a sign follows by considering $|u_\alpha|$, which is also an extremal, also satisfying \eqref{eqpert0} hence by the maximum principle it never vanishes. In particular also $u_\alpha$ never vanishes, and by continuity it has a sign.

Finally, if $\Omega =B_1$, the claim about the symmetry of $u_\alpha$ follows at once by choosing $u_k$ radially symmetric and decreasing, which is possible by symmetrization.
\end{proof}

\paragraph{Proof of Theorems \ref{trmMoser} and \ref{trmCC} assuming Theorem \ref{trmEnergy}, and of Corollary \ref{corCC} using Theorem \ref{trmEnergypert}.} Set $\alpha_k=4\pi-\frac1k$ and let $u_k=u_{\alpha_k}>0$ be the radially symmetric extremal to $(I_{\alpha_k}^g)$ with $\Omega= B_1$ given by Proposition \ref{propsubcrit2}. According to \eqref{eqEnergypert} we have
$$\limsup_{k\to\infty} u_k(0)=\limsup_{k\to\infty} \max_{B_1} u_k <\infty,$$
otherwise for some $k$ large enough we would have
$$4\pi+\frac{4\pi}{u_k^4(0)}+o(u_k^{-4}(0))\le \|\nabla u_k\|_{L^2}^2=4\pi-\frac{1}{k},$$
which is a contradiction.
Then $u_k(0)=\max_{B_1}|u_k|\le C$ and by elliptic estimates we have $u_k\to u_\infty$ in $C^1(\bar \Omega)$.
%$C^\ell(\bar \Omega)$ for every $\ell\ge 0$.
It is now easy to see that $u_\infty$ is an extremal for $(I_{4\pi}^g)$. Indeed
$$\int_{B_1}(1+g(u_k))e^{u_k^2}dx\uparrow \int_{B_1}(1+g(u_\infty))e^{u_\infty^2}dx,\quad \text{as }k\to\infty,$$
and if there was a function $v\in H^1_0(B_1)$ with $\|\nabla v\|_{L^2}^2\le 4\pi$ and
$$\int_{B_1}(1+g(v))e^{v^2}dx>\int_{B_1}(1+g(u_\infty))e^{u_\infty^2}dx,$$
we could find (for instance by monotone  convergence) $k$ large such that $\tilde u_k:=\sqrt{\frac{\alpha_k}{4\pi}}v$ satisfies
$$\int_{B_1}(1+g(\tilde u_k))e^{\tilde u_k^2}dx>\int_{B_1}(1+g(u_\infty))e^{u_\infty^2}dx\ge \int_{B_1}(1+g(u_k))e^{u_k^2},$$
which would contradict the maximality of $u_k$, since $\|\nabla \tilde u_k\|_{L^2}^2\le\alpha_k$. Then $u_\infty$ is an extremal $(I_{4\pi}^g)$. This also implies $(I_{4\pi}^g)$ (hence \eqref{MT}) for $\Omega=B_1$, and by symmetrization and scaling, also for any domain $\Omega$ with finite measure.
This completes the proof. \hfill$\square$

%\begin{rmk}
%This proof, also works for perturbed functionals. For a given function $g:\R\rightarrow \R$ satisfying \eqref{condg}, \eqref{defh} and \eqref{condh1},  as in Proposition \ref{propsubcrit2} one proves that for any $\alpha<4\pi$ the subcritical inequality 
%\begin{equation*} 
%\sup_{u\in H^1_0(B_1):\|\nabla u\|_{L^2}^2\le \alpha}\int_{B_1}(1+g(u))e^{u^2}dx<\infty,
%\end{equation*}
%admits a positive radially symmetric extremal $u_\alpha$ solution of \eqref{eqMTpert} (here we use that the function $(1+g(t))e^{t^2}$ is increasing for $t>0$, which follows from \eqref{defh}-\eqref{condh1}). In particular the existence of extremals  for \eqref{MTpert} in Theorem \ref{trmEnergypert} can be deduced from the energy estimate   \eqref{eqEnergypert} as in the above proof.
%\end{rmk}

\begin{lemma}\label{lemmafkL} Let $|\Omega|<\infty$, and consider a sequence of non-negative functions $(f_k)\subset L^1(\Omega)$ with $f_k\to f$ a.e. and with
\begin{equation}\label{fkL}
\int_{\{f_k>L\}} f_k dx =o(1),
\end{equation}
with $o(1)\to 0$ as $L\to\infty$ uniformly with respect to $k$. Then $f_k\to f$ in $L^1(B_1)$.
\end{lemma}
\begin{proof}
By Fatou's lemma \eqref{fkL} implies $f \in L^1(\Omega)$. From the dominated convergence theorem
$$\min\{f_k, L\}\to \min\{f, L\}\quad \text{in }L^1(\Omega),$$
and the convergence of $f_k$ to $f$ in $L^1$ follows at once from \eqref{fkL} and the triangle inequality.
\end{proof}

\section{Proof of Theorem \ref{trmEnergy}}\label{sec3}

Let $u_k$ and $\mu_k=u_k(0)\to\infty$ be as in Theorem \ref{trmEnergy}. % Since $\Delta u_k\le 0$
In order to estimate $\|\nabla u_k\|_{L^2}$, after a well-known scaling (see \eqref{defetak} below) we reduce to study a function $\eta_k$ which solves a perturbed version of the Liouville equation, namely \eqref{eq:etak}. We will make a Taylor expansion of the right-hand side of \eqref{eq:etak} up to order $\mu_k^{-6}$ (Lemma \ref{taylor}) and expand $\eta_k=\eta_0+\frac{w_0}{\mu_k^2}+\frac{z_0}{\mu_k^4}+\frac{\phi_k}{\mu_k^6}$. Inspired from  \cite{MM} (where the Taylor expansion was made only up to order $\mu_k^{-4}$), we will prove uniform bounds on the error term $\phi_k$ up to sufficiently large scales. This can be achieved by ODE theory and a fixed point argument, see Lemma \ref{lemmaerror}. Together with the asymptotic behaviour of $w_0$, which is explicit thanks to Lemma \ref{lemma:w}, this implies
$$\|\nabla u_k\|_{L^2}^2= 4\pi +O(\mu_k^{-4}),$$
but with no information about the sign of the error $O(\mu_k^{-4})$. In order to obtain the more precise estimate \eqref{eqEnergy} we shall need the asymptotic behaviour of the function $z_0$, which is not given by an explicit formula. For this we will use the somewhat surprising Lemma \ref{lemmamagic} (also see Corollary \ref{corz0}).

\subsection{Taylor expansions and behaviour at large scales}

%$$2\pi u'(r)=\int_{B_r}\Delta udx < 0; \qquad r > 0,$$
%hence $u(r)$ is decreasing. Let us assume
%\begin{equation}\label{ublow}
%\mu_k:=u_k(0)=\max_{B_1} u_k \to \infty\quad \text{as }k\to\infty.
%\end{equation}
We will start with the following standard blow-up procedure. Set  $r_k>0$ such that
\begin{equation}\label{defrk}
r_k^2 \lambda_k \mu_k^2 e^{\mu_k^2}=4
\end{equation}
and rescale $u_k$ to a new function $\eta_k$ defined on $B_{r_k^{-1}}$ as
\begin{equation}\label{defetak}
\eta_k(x):=\mu_k (u_k(r_k x)-\mu_k).
\end{equation}
Notice that
\begin{equation}\label{eq:etak}
\left\{\begin{array}{ll}
-\Delta \eta_k =4\bra{1+\frac{\eta_k}{\mu_k^2}} e^{2\eta_k+\frac{\eta_k^2}{\mu_k^2}}&\text{in }[0,r_k^{-1}]\\
\eta_k(0)=\eta_k'(0)=0\rule{0cm}{.6cm}&
\end{array}
\right.
\end{equation}
and, as $\mu_k\to\infty$, the nonlinearity on the right-hand side approaches $4e^{2\eta_k}$. More precisely one has:

\begin{lemma}[\cite{Dru,MM}]\label{etamu2} Let $r_k$, $\eta_k$ be as in \eqref{defmuk}, \eqref{defrk} and \eqref{defetak}, with $\eta_k$ solving \eqref{eq:etak}. Then as $k\to\infty$ we have $r_k\to 0$,
\begin{equation}\label{etamu}
\eta_k(x) \to \eta_0(x):= -\log(1+|x|^2)\quad \text{in }C^1_{\loc}(\R^2),
\end{equation}
%and
%\begin{equation}\label{enerbolla}
%\lim_{R\to\infty}\lim_{k\to\infty}\int_{B_{Rr_k}}\lambda_k u_k^2 e^{u_k^2}dx=\int_{\R^2}4e^{2\eta_0}dx=4\pi.
%\end{equation}
and $\eta_0$ solves
\begin{equation}\label{eqliou}
-\Delta \eta_0=4e^{2\eta_0}\quad \text{in }\R^2.
\end{equation}
Moreover
\begin{equation}\label{ener1}
\lim_{R\to\infty}\lim_{k\to\infty}\int_{B_{Rr_k}}\lambda_k u_k^2 e^{u_k^2}dx=\int_{\R^2}4e^{2\eta_0}dx=4\pi.
\end{equation}
\end{lemma}

One easily sees that \eqref{ener1} implies
\begin{equation}\label{liminf4pi}
\liminf_{k\to\infty}\|\nabla u_k\|_{L^2}^2\ge 4\pi.
\end{equation}
In order to improve \eqref{liminf4pi} to
\begin{equation}\label{lim4pi}
\lim_{k\to\infty}\|\nabla u_k\|_{L^2}^2=4\pi,
\end{equation}
in \cite{MM} Malchiodi and the second author investigated the blow-up behaviour of the sequence $u_k$ 
%at scales larger than $r_k$ and 
up to a higher order of precision. %by Taylor-expanding the right-hand side of \eqref{eq:etak}. % and studied the equation satisfied that $\eta_k-\eta_0$, suitably scaled.

\begin{lemma}[\cite{MM}]\label{lemma:w} Set $w_k:=\mu_k^2(\eta_k-\eta_0).$ Then we have $w_k\to w_0$ in $C^1_{\loc}(\R^2),$
where
\begin{equation}\label{defw}
\begin{split}
w_0(r)& := \eta_0(r) + \frac{2r^2}{1+r^2}-\frac{1}{2}\eta_0^2(r)+\frac{1-r^2}{1+r^2}\int_1^{1+r^2} \frac{\log t}{1-t}dt
\end{split}
\end{equation}
is the unique solution to the ODE
\begin{equation}\label{eqw}
\left\{
\begin{array}{l}
-\Delta w_0= 4e^{2\eta_0}(\eta_0+\eta_0^2+2w_0)\text{ in }\R^2\\
w_0(0)=w_0'(0)=0.\rule{0cm}{0.5cm}
\end{array}
\right.
\end{equation}
Moreover $w_0(r)=\eta_0(r)+O(1)$ as $r\to\infty$ and in fact
\begin{equation}\label{Deltaw0}
\int_{\R^2}\Delta w_0 dx=-4\pi.
\end{equation}
%\begin{equation}\label{w0'}
%w_0'(r)=-\frac{2}{r}+O(r^{-3}\log^2 r), \quad {as }r\to\infty.
%\end{equation}
\end{lemma}
%A formal expansion of $\|\nabla u_k\|_{L^2}^2$, together with Lemma \ref{lemma:w} suggests that
%$$\|\nabla u_k\|_{L^2}^2= 4\pi +o(\mu_k^{-2}),$$
%but no information can be extracted about the sign of the term $o(\mu_k^{-2})$. In others words, in the expansion of $\|\nabla u_k\|_{L^2}^2$, the coefficient of the term in $\mu_k^{-2}$ turns out to be $0$.

One can further prove that 
\begin{equation}\label{derw0}
w_0'(r)= -\frac{2}{r} +O\bra{\frac{\log^2{r}}{r^3}}
\end{equation}
as $r\to \infty$, which will be important in our analysis. This follows from the explicit expression \eqref{defw} but can also be deduced from the structure of equation \eqref{eqw}, see Corollary \ref{corw0}.

\medskip

To prove Theorem \ref{trmEnergy} we need to further expand the right-hand side of \eqref{eq:etak}, namely we write
$$\eta_k=\eta_0+\frac{w_0}{\mu_k^2}+\frac{z_k}{\mu_k^4},$$
for an unknown (locally bounded) error  $z_k$, and formally compute
\[
\begin{split}
-\Delta \eta_k &=4\bigg(1+\frac{\eta_k}{\mu_k^2}\bigg) e^{2\eta_k+\frac{\eta_k^2}{\mu_k^2}}\\
&=4e^{2\eta_0}\bigg[1+\frac{\eta_0+\eta_0^2+2w_0}{\mu_k^2}+  \frac{w_0+2w_0^2+4\eta_0 w_0+2w_0 \eta_0^2+\eta_0^3+\frac{1}{2}\eta_0^4+2z_k}{\mu_k^4}\bigg]\\
&\quad +O(\mu_k^{-6}).
\end{split}
\]
This suggests to define $z_0$ as the only radial solution to the Cauchy problem
\begin{equation}\label{eqz}
\left\{
\begin{array}{l}
-\Delta z_0= 4e^{2\eta_0}(w_0+2w_0^2+4\eta_0 w_0+2w_0 \eta_0^2+\eta_0^3+\frac{1}{2}\eta_0^4+2z_0) \text{ in }\R^2\\
z_0(0)= z_0'(0)=0.\rule{0cm}{0.5cm}
\end{array}
\right.
\end{equation}
Even though we do not have an explicit formula for $z_0$, we will show
%Lemma \ref{lemmalog} will imply
\begin{equation}\label{zlog}
z_0(r) =\beta \log(r) +O(1),\quad \text{as }r\to\infty,
\end{equation}
for some constant $\beta$. In fact
%, using Lemma \ref{lemmamagic} 
we will prove %(see Corollary \ref{corz0}) 
\begin{equation}\label{intz_0}
\beta=\frac{1}{2\pi}\int_{\R^2}\Delta z_0dx=-6-\frac{\pi^2}{3},
\end{equation}
which will be crucial in the proof of Proposition \ref{stimabasso}. To simplify our exposition of the proof, we postpone the analysis of the asymptotic behaviour of $z_0$ to the end of the section, see Lemmas \ref{lemmalog}, \ref{lemmamagic} and Corollary \ref{corz0}. %,  To prove the next lemma \eqref{zlog} suffices.

\medskip 
The problem now is to use $\eta_0$, $w_0$ and $z_0$ to approximate $\eta_k$ in a good sense (up to error $O(\mu_k^{-6}\log^6 r)$) and for sufficiently large radii. For this we will use a method inspired from \cite[Lemma 5]{MM}.

\begin{lemma}\label{taylor}
Let $0\le S\le s_k\le e^{\mu_k}$ and $\phi:[S,s_k]\to \R$ be given so that $\phi=o(\mu_k^6)$ uniformly on $[S,s_k]$.
Set
$$\eta:=\eta_0+\frac{w_0}{\mu_k^2}+\frac{z_0}{\mu_k^4}+\frac{\phi}{\mu_k^6}$$
and
\begin{equation}\label{eq:phi}
\Phi_k(r,\phi):=\mu_k^6\left[4\left(1+\frac{\eta}{\mu_k^2}\right)e^{2\eta+\frac{\eta^2}{\mu_k^2}} +\Delta \eta_0+\frac{\Delta w_0}{\mu_k^2}+\frac{\Delta z_0}{\mu_k^4}\right].
\end{equation}
Then
$$
\Phi_k(r,\phi) =  4 e^{2\eta_0}\left(2\phi +o(1)\phi+ O(\mu_k^{-2}\xi^2)\phi+ O(\xi^6) \right),\quad \text{uniformly for } r\in [S,s_k],
$$
where
\begin{equation}\label{defxi}
\xi(r):=1+\log(1+r).
\end{equation}
\end{lemma}

\begin{proof} 
By Lemma \ref{etamu2}, Lemma \ref{lemma:w} and \eqref{zlog} we have $|\eta_0|+|w_0|+|z_0|=O(\xi)$. Moreover the assumptions on $s_k$ imply $\mu_k^{-1} \xi =O(1)$ uniformly on $[0,s_k]$. This will be used several times throughout the proof.
In order to expand the exponential term in $\Phi_k(r,\phi)$ we write
\begin{equation}\label{defphik0}
\begin{split}
\psi:= & \; 2\eta+\frac{\eta^2}{\mu_k^2} -2\eta_0\\
= &  \; \frac{2w_0+\eta_0^2}{\mu_k^2}+\frac{2z_0+2\eta_0w_0}{\mu_k^4} +\frac{2\phi}{\mu_k^6} +%O(\mu_k^{-2}\xi)
o(1)\frac{\phi}{\mu_k^6}+
O(\mu_k^{-6}\xi^2)
\end{split}
\end{equation}
uniformly in $[S, s_k]$.
%where
%$$|O(\delta_k \log^2 r)| \le C\delta_k(1+ \log^2(1+r)),\quad \text{for } r\in  [0,e^{\mu_k}],$$
%for a constant $C$ not depending on $C_1$ in \eqref{phiC1} and for $k$ large enough (depending on $C_1$). Indeed all the terms in $\psi$ containing $\phi$, except for $\frac{2\phi}{\mu_k^6}$, are bounded as
%$$O(C_1^2\mu_k^{-8}(1+\log^2(1+r))),$$
%and for $k$ large enough we can assume that $\frac{C_1^2}{\mu_k^2}\le 1$.
Similarly
\[\begin{split}
\psi^2&= \frac{4w_0^2 + 4w_0\eta_0^2+ \eta_0^4}{\mu_k^4} + o(1)\frac{\phi}{\mu_k^6}+O(\mu_k^{-2}\xi^2)\frac{\phi}{\mu_k^6}+O(\mu_k^{-6}\xi^4),\\
\psi^3&= O(\mu_k^{-6}\xi^6)+o(1)\frac{\phi}{\mu_k^6}+O(\mu_k^{-4}\xi^4)\frac{\phi}{\mu_k^6}.
\end{split}\]
From \eqref{defphik0} we easily get that $\psi_k$ is uniformly bounded for $r\in [S,s_k]$ and we can write
\[\begin{split}
e^{\psi}-1 -\psi -\frac{\psi^2}{2} & =O(\max\{1,e^{\psi}\} ) \;\psi^3 \\
&= O(\mu_k^{-6}\xi^6)+o(1)\frac{\phi}{\mu_k^6}+O(\mu_k^{-4}\xi^4)\frac{\phi}{\mu_k^6}.
\end{split}\]
Therefore
\[\begin{split}
e^{\psi}&= 1+ \frac{2w_0+\eta_0^2}{\mu_k^2} +\frac{2z_0+2w_0^2 +2\eta_0w_0+2w_0 \eta_0^2+\frac{1}{2}\eta_0^4}{\mu_k^4} \\
&\quad + \frac{2\phi}{\mu_k^6} + o(1)\frac{2\phi}{\mu_k^6}+
O(\mu_k^{-2}\xi^2)\frac{\phi}{\mu_k^6}+O(\mu_k^{-6}\xi^6).
\end{split}\]
To obtain the Taylor expansion of $\Phi_k(r, \phi)$, we also need to multiply this term by
\[
1+\frac{\eta}{\mu_k^2}=1+\frac{\eta_0}{\mu_k^2}+ \frac{w_0}{\mu_k^4}+ O(\mu_k^{-6}\xi)+ o(1)\frac{\phi}{\mu_k^6},
\]
and finally, using \eqref{eqliou}, \eqref{eqw} and \eqref{eqz} we obtain
$$
\Phi_k(r, \phi) =  4 e^{2\eta_0}\left(2\phi +o(1)\phi+O(\mu_k^{-2}\xi^2)\phi+ O(\xi^6) \right),
$$
as was to be shown.
\end{proof}

\begin{prop}\label{lemmaerror} There exist $M>0$ and $T>0$ such that
$$\eta_k=\eta_0+\frac{w_0}{\mu_k^2}+\frac{z_0}{\mu_k^4}+\frac{\phi_k}{\mu_k^6}$$
with
\begin{equation}\label{stimaphik}
|\phi_k(r)|\le M\xi(r), \quad \text{for }r\in [0,e^{\mu_k}], \quad |\phi_k'(r)|\le \frac{M}{r},\quad \text{for }r\in [T,e^{\mu_k}]
\end{equation}
for $k$ large (depending on $M$ and $T$), where $\xi$ is as in \eqref{defxi}.
\end{prop}

\begin{proof} This follows from a fixed-point argument and the uniqueness of solutions of ODEs.

From Lemma \ref{etamu2} we have that for every interval $[0,T]$, $\phi_k=o(\mu_k^6)$ uniformly in $[0,T]$, hence by Lemma \ref{taylor}
\begin{equation}\label{eqphik}
\left\{
\begin{array}{l}
-\Delta \phi_k=\Phi_k(r,\phi_k) = 4e^{2\eta_0}(2\phi_k+o(1)\phi_k+O(1))\\
\phi_k(0)= \phi_k'(0)=0\rule{0cm}{0.5cm}
\end{array}
\right.
\end{equation}
with $o(1)\to 0$ and $|O(1)|\le C$ uniformly in $[0,T]$, and from ODE theory it follows that $\phi_k$ is uniformly bounded in $[0,T]$. In particular there exists a constant $C(T)$ such that 
\begin{equation}\label{eq:phiT}
|\phi_k(r)| \leq C(T), \quad |\phi_k'(r)| \leq C(T), \quad \text{for }r\in [0,T],
\end{equation}
uniformly in $k$.

%Define the norm
%$$
%  \|f\| = \sup_{r \in (T,e^{\mu_k}]} \left| \frac{f(r)}{\log r
%  - \log T} \right|.$$
For a large constant $M > 0$ to be fixed later, we will work with the following
set of functions
$$
  \mathcal{B}_M = \left\{ \phi\in C^0([T,e^{\mu_k}]) :  \sup_{r \in [T,e^{\mu_k}]} \frac{|\phi(r) - \phi_k(T)|}{\log r}
  \leq M\right\},
$$
seen as a convex closed subset of $C^0([T,e^{\mu_k}])$.
Notice that for $\phi\in \M{B}_M$ we have
\begin{equation}\label{eqphilog}
|\phi(r)|\le |\phi_k(T)|+|\phi(r)-\phi_k(T)|\le C(T)+M\log r
\end{equation}
for any  $r\in [T,e^{\mu_k}]$.
In particular 
$$\frac{|\phi|}{\mu_k^6}\le \frac{C(T)+M\mu_k}{\mu_k^6}=o(1)$$
uniformly on $[T,e^{\mu_k}]$ for $k$ large enough. Then, by  Lemma \ref{taylor},  we have
\begin{equation}\label{eqPhi}
\Phi_k(r,\phi)= 4e^{2\eta_0}(2\phi+ O(\xi^6)),
\end{equation}
where $|O(\xi^6)|\le C\xi^6$ uniformly in $[T,e^{\mu_k}]$ for $k\ge k_0(T,M)$ sufficiently large.

Let now $F_k:\M{B}_M\to C^0([T,e^{\mu_k}])$ (for a fixed $k$) associate to a function $\phi$ the solution $\bar \phi$ of 
\begin{equation}\label{eqphibar}
\left\{
\begin{array}{ll}
-\frac{1}{r}(r\bar \phi'(r))'= \Phi_k(r,\phi(r)) &\text{for }T\le r\le e^{\mu_k}\\
\bar\phi(T)=\phi_k(T)\rule{0cm}{0.6cm}\\
\bar\phi'(T)=\phi'_k(T). \rule{0cm}{0.6cm}
\end{array}
\right.
\end{equation}
We will show that $F_k$ sends $\mathcal{B}_M$ into itself for suitable choices of $M$ and $T$, and is compact.
Indeed for $\phi \in \mathcal{B}_M$ one can integrate \eqref{eqphibar} and use \eqref{eq:phiT}-\eqref{eqPhi} to get
\[\begin{split}
|r\bar\phi'(r)|&=\left|T\phi_k'(T)-\int_T^r t\Phi_k(t,\phi(t))dt\right|\\
&\le TC(T)+ \int_T^r \frac{8t|\phi(t)|}{(1+t^2)^2}dt+\int_T^r\frac{4Ct\xi^6(t)}{(1+t^2)^2}dt\\
&\le (T+o_T(1))C(T)+ Mo_T(1)+o_T(1),
\end{split}
\]
where
$$|o_T(1)|\le \int_T^\infty \frac{8t (1+\log t)}{(1+t^2)^2} dt  +\int_T^\infty \frac{4C t\xi^6(t)}{(1+t^2)^2} dt \to 0,\quad \text{as }T\to\infty.$$ 
First choosing $T$ so large that $|o_T(1)|\le \frac{1}{2}$, and then $M$ such that
\begin{equation}\label{defTM}
\bra{T+\frac12}C(T)+\frac{1}{2}\le \frac{M}{2},
\end{equation}
we obtain
\begin{equation}\label{boundphi'}
|r\bar\phi'(r)|\le M,\quad \text{for }T\le r\le e^{\mu_k}.
\end{equation}
Integrating again we infer
%\[
%\begin{split}
%|\bar \phi (r)-\phi_k(T)|&\le \int_T^r |\bar \phi'(t)|dt\\
%&\le \int_T^r \frac{(T+o_T(1))C(T)+(M+1)o_T(1)}{t}dt\\
%&= ((T+o_T(1))C(T)+(M+1)o_T(1))(\log r-\log T).
%\end{split}
%\]
%we conclude that
\begin{equation}\label{boundphi}
|\bar \phi (r)-\phi_k(T)| \le M (\log r-\log T)\le M\log r,
\end{equation}
hence $\bar \phi\in \M{B}_M$. Then $F_k$ sends $\M{B}_M$ into itself. Moreover it is compact with respect to the uniform convergence by the theorem of Ascoli-Arzel\`a, since for any sequence $(\psi_n)\subset \M{B}_M$, the sequence $(F_k(\psi_n))$ is uniformly bounded and equicontinuous by \eqref{boundphi'} and \eqref{boundphi}. %we have
%$$\bar \psi_n(r)' =\bar\psi_n'(\log r-\log T)+\frac{\bar \psi_n(r)}{r}$$
%which is uniformly bounded on $[T,e^{\mu_k}]$ by \eqref{boundphi'}-\eqref{boundphi}, so that up to a subsequence
%$$\bar \psi_n(r) (\log r-\log T)\to \bar \psi_\infty(r) (\log r-\log T),\quad \text{uniformly},$$
%i.e. $\|\bar \psi_n-\psi_\infty\|\to 0$ for some $\psi_\infty\in \M{B}_M$.
Therefore, by the fixed-point theorem of Caccioppoli-Schauder (see e.g. \cite[Corollary 11.2]{GT}) $F_k$ has a fixed point $\phi\in \M{B}_M$, which solves \eqref{eqphik}. Then, by uniqueness for the Cauchy problem, we have $\phi=\phi_k$ in $[T,e^{\mu_k}]$, whence the bounds
\begin{equation}\label{stimaphikbis}
|\phi_k(r)|\le C(T)+M\log r,\quad \text{for } T\le r\le e^{\mu_k},
\end{equation}
which is another way of writing the first inequality in \eqref{stimaphik} (a priori the identity $\phi=\phi_k$ holds as long as $\phi_k$ is defined, i.e. up to $r_k^{-1}$; on the other hand, the reader can easily verify that $\eta_k>-\mu_k^2$ as long \eqref{stimaphikbis} holds, so that in particular $r_k^{-1}> e^{\mu_k}$).  The second inequality in \eqref{stimaphik} follows from \eqref{boundphi'}.
\end{proof}

\subsection{Proof of Theorem \ref{trmEnergy} completed}

We are now in a position to use the Taylor expansion computed in the previous section to estimate the Dirichlet energy of $u_k$.

\begin{prop}\label{stimabasso} Given a sequence $(s_k)$ with $s_k\in [\mu_k^p,e^{\mu_k}]$ for some $p>2$, we have
\begin{equation}\label{eqstimabasso}
\int_{B_{r_ks_k}} \lambda_k u_k^2e^{u_k^2}dx= 4\pi+\frac{4\pi}{\mu_k^4}+o(\mu_k^{-4}).
\end{equation}
\end{prop}

\begin{proof} We start writing
\[
\begin{split}
(I)&:=\int_{B_{r_k s_k}}\lambda_k u_k^2e^{u_k^2}dx=4\int_{B_{s_k}} \bra{1+\frac{\eta_k}{\mu_k^2}}^2 e^{2\eta_k+\frac{\eta_k^2}{\mu_k^2}}dx\\
&= \int_{B_{s_k}}\left(1+\frac{\eta_k}{\mu_k^2}\right) \left(-\Delta \eta_0 - \frac{\Delta w_0}{\mu_k^2} - \frac{\Delta z_0}{\mu_k^4 } +\frac{\Phi_k(r,\phi_k)}{\mu_k^6} \right)dx,
\end{split}
\]
where $\Phi_k$ is as in \eqref{eq:phi}.
Using Lemma \ref{taylor} and Proposition  \ref{lemmaerror} we have on $[0,s_k]$
$$1+\frac{\eta_k}{\mu_k^2}= 1+\frac{\eta_0}{\mu_k^2}+\frac{w_0}{\mu_k^4}+O(\mu_k^{-5}), $$
and
$$\Phi_k(r, \phi_k)=O(e^{2\eta_0}\xi^6),$$
where $\xi$ is as in \eqref{defxi}. In particular
$$\int_{B_{s_k}} |\Phi_k(r, \phi_k)|dx\le C\int_{\R^2}\frac{\xi^6(x)}{(1+|x|^2)^2}dx\le C.$$
Similarly
$$\max\{ |\Delta \eta_0|,|\Delta w_0|,|\Delta z_0|\}=O(e^{2\eta_0}\xi^4),$$
so that
$$\int_{B_{s_k}}\xi\max\{ |\Delta \eta_0|,|\Delta w_0|,|\Delta z_0|\} dx\le C\int_{\R^2}\frac{\xi^5(x)}{(1+|x|^2)^2}dx\le C.$$
Summing up one gets
\[\begin{split}
(I)&=\int_{B_{s_k}}\left(-\Delta\eta_0-\frac{\eta_0\Delta\eta_0+\Delta w_0}{\mu_k^2}-\frac{w_0\Delta\eta_0+\eta_0\Delta w_0+\Delta z_0}{\mu_k^4} \right)dx + O(\mu_k^{-5})\\
&=:(I_0)+\frac{(I_2)}{\mu_k^2}+\frac{(I_4)}{\mu_k^4}+O(\mu_k^{-5}).
\end{split}\]
Now we compute
$$(I_0)=\int_{B_{s_k}} 4e^{2\eta_0}dx=4\pi\left(1-\frac{1}{1+s_k^2}\right)=4\pi+o(\mu_k^{-4}).$$
Using the divergence theorem, and  \eqref{derw0} we get
\[\begin{split}
(I_2)&= \int_{B_{s_k}}4e^{2\eta_0}\eta_0dx-2\pi s_k w_0'(s_k)\\
&= 4\pi\left( \frac{\log(1+s_k^2)}{1+s_k^2}+\frac{1}{1+s_k^2}-1\right)+4\pi+O(s_k^{-2}\log^2 s_k)\\
&= o(\mu_k^{-4}).
\end{split}\]
From \eqref{intz_0} we get 
$$
-\int_{B_{s_k}} \Delta z_0 dx = 2\pi \left(6+\frac{\pi^2}{3}\right) + o(1),
$$
while a direct computation shows that
$$
-\int_{B_{s_k}} \bra{w_0 \Delta \eta_0 +\eta_0 \Delta w_0} dx = 4\int_{B_{s_k}} e^{2\eta_0} (w_0+\eta_0^2+\eta_0^3+2 w_0\eta_0)dx = -8\pi - \frac{2}{3}\pi^3 +o(1),
$$
hence $(I_4) =4\pi +o(1)$, and we conclude by summing up. 
\end{proof}

\begin{rmk}
The freedom in the choice of the sequence $s_k\in [\mu_k^p,e^{\mu_k}]$  in Proposition \ref{stimabasso} implies that
$$\int_{B_{r_k e^{\mu_k}}\setminus B_{r_k \mu_k^p}} \lambda_k u_k^2e^{u_k^2}dx=o(\mu_k^{-4})$$
for any $p>2$.
\end{rmk}

\begin{OP}
Is there any geometric meaning to the term $\frac{4\pi}{\mu_k^4}$ in \eqref{eqstimabasso}, in particular to its positivity?
\end{OP}
From Lemma \ref{etamu2} we know that the first $4\pi$ appearing on the right-hand side of \eqref{eqstimabasso} can be seen as the area of $S^2$, since $-\Delta \eta_0$ is the conformal factor of the pull-back of the metric of $S^2$ onto $\R^2$ via stereographic projection. The second $4\pi$ appearing in \eqref{eqstimabasso} depends on the asymptotic behavior of $z_0$, but we do not have a geometric interpretation.

\medskip

While Proposition \ref{stimabasso} gives a lower bound on $\|\nabla u_k\|_{L^2}$, we will now prove an upper bound. First of all we shall observe $\eta_k(r)\le \eta_0(r)$ for sufficiently large $r$, which was proved in \cite{MM}.    The next lemma gives a more general statement which will turn out to be useful also in the next sections. 

\begin{lemma}\label{etaketa0}
Let $\bar{\eta}_k:[0,r_k^{-1}]\ra \R$ be a sequence of $C^2$ functions satisfying $\Delta \bar \eta_k \le 0$. Assume further that 
$\bar \eta_k$ has an expansion of the form 
\begin{equation}\label{expgen}
\bar \eta_k = \eta_0 + \frac{w}{\mu_k^2} +\psi_k \qquad \text{in } [0,\mu_k^2], 
\end{equation}
with $w: [0,+\infty)\ra \R$, $ \psi_k :[0,r_k^{-1}) \rightarrow \R$ satisfying 
\begin{equation}\label{wneg}
w(\mu_k^2) \le -1,  
\end{equation}  
\begin{equation}\label{wlapneg}
 \int_{\R^2} \Delta  w \;dx <0,
 \end{equation} 
 and 
 \begin{equation}\label{errpsik}
 \sup_{[0,\mu_k^2]} |\psi_k| + \int_{B_{\mu_k^2}} |\Delta \psi_k| dx  =o(\mu_k^{-2}).
 \end{equation}
 Then $\bar \eta_k\le \eta_0$ in $[\mu_k^2, r_k^{-1}]$, for  $k$ sufficiently large.
\end{lemma}

\begin{proof}
By \eqref{expgen}, \eqref{wlapneg} and \eqref{errpsik} we compute
\[
\begin{split}
 \int_{B_{\mu_k^2}} \Delta  \bar \eta_k dx &=  \int_{B_{\mu_k^2}} \Delta  \eta_0 dx + \frac{1}{\mu_k^2} \int_{B_{\mu_k^2}} \Delta w dx + o(\mu_k^{-2}) \\
 & =-4\pi + \frac{1}{\mu_k^2} \int_{\R^2} \Delta w dx + o(\mu_k^{-2})<-4\pi.
\end{split}
\]
Since $\Delta \bar \eta_k\le 0$, for $r\in [\mu_k^2,r_k^{-1}]$ we get
$$ 
\int_{B_{r}} \Delta  \bar \eta_k dx  \le   \int_{B_{\mu_k^2}} \Delta  \bar \eta_k dx <-4\pi < \int_{B_r} \Delta \eta_0 dx
$$
and by the divergence theorem we deduce $\bar \eta_k'(r)\le \bar \eta_0'(r)$. Finally \eqref{expgen}, \eqref{wneg} and \eqref{errpsik} guarantee that $\bar \eta_k (\mu_k^2) \le  \eta_0(\mu_k^2)$ for large $k$,  and the conclusion follows from the fundamental theorem of calculus.
\end{proof}

Clearly, by \eqref{Deltaw0} and Proposition \ref{lemmaerror}, Lemma \ref{etaketa0} applies to $\eta_k$.

\begin{prop}\label{stimaalto}
For some $p>2$ let $s_k\in [\mu_k^p, e^{\mu_k}]$. Then we have
$$\int_{B_1\setminus B_{s_kr_k}}\lambda_k u_k^2 e^{u_k^2}dx\le \frac{2\pi}{\mu_k^4}+o(\mu_k^{-4}),\quad \text{as }k\to\infty.$$
\end{prop}

\begin{proof} With the usual scaling,  we have to prove that
$$(I):=4\int_{B_\frac{1}{r_k}\setminus B_{s_k}}  \bra{1+\frac{\eta_k}{\mu_k^2}}^2 e^{2\eta_k+\frac{\eta_k^2}{\mu_k^2}}dx\le \frac{2\pi}{\mu_k^4}+o(\mu_k^{-4}).$$
By Lemma \ref{etaketa0} for $r\in [s_k, r_k^{-1}]$ and for $k$ large enough we have $\eta_k\le \eta_0$. Let us set $t_k:= \sqrt{e^{\mu_k^2}-1}$ and $\tilde{t}_k:=\sqrt{\mu_k^{2p}-1}$. We claim that, for $k$ large enough, $\frac{1}{r_k} \le t_k$. Otherwise, as soon as $t_k\ge e^{\mu_k}$, we would have
$$
u_k(r_k t_k) = \mu_k +\frac{\eta_k(t_k)}{\mu_k}\le \mu_k + \frac{\eta_0(t_k)}{\mu_k} =0,
$$ 
which contradicts the positivity of $u_k$ in $B_1$. Hence 
$$(I) \le  \int_{B_{t_k} \setminus B_{\tilde{t}_k}} \bra{1+\frac{\eta_0}{\mu_k^2}}^2 e^{2\eta_0 +\frac{\eta_0^2}{\mu_k^2}} dx =2\pi \int_{\tilde{t}_k}^{t_k}
 r \bra{1+\frac{\eta_0}{\mu_k^2}}^2 e^{2\eta_0 +\frac{\eta_0^2}{\mu_k^2}} dr =:(II).$$
With the changes of variable $s=-\eta_0(r)=\log(1+r^2)$ and $\tau=\frac{s}{\mu_k}-\frac{\mu_k}{2}$, we get
\begin{equation}\label{extint}
\begin{split}
(II)%&=\pi \int_{2p\log\mu_k}^{\mu_k^2}
% \bra{1-\frac{s}{\mu_k^{2}}}^2 e^{-s +\frac{s^2}{\mu_k^2}} ds \\
 &= \pi e^{-\frac{\mu_k^2}{4}} \int_{2p\log\mu_k}^{\mu_k^2}
 \bra{1-\frac{s}{\mu_k^2}}^2 e^{(\frac{s}{\mu_k}-\frac{\mu_k}{2})^2} ds \\
% &= \pi \mu_k e^{-\frac{\mu_k^2}{4}} \int_{2-\frac{\mu_k}{2}}^{\frac{\mu_k}{2}}
%\bra{\frac{1}{2}-\frac{\tau}{\mu_k}}^2 e^{\tau^2} d\tau\\
 & =  \pi  e^{-\frac{\mu_k^2}{4}} \int_{\frac{2p\log\mu_k}{\mu_k}-\frac{\mu_k}{2}}^{\frac{\mu_k}{2}}
 \left(\frac{\mu_k}{4}-\tau +\frac{\tau^2}{\mu_k}\right) e^{\tau^2} d\tau.
\end{split}
\end{equation}
Since $p>2$ we have 
\[
-e^{-\frac{\mu_k^2}{4}} \int_{\frac{2p\log\mu_k}{\mu_k}-\frac{\mu_k}{2}}^{\frac{\mu_k}{2}}
 \tau  e^{\tau^2} d\tau =-\frac{1}{2}+o(\mu_k^{-4}).
\]
Moreover it is simple to verify (using e.g. de l'H\^opital rule) that
\[\begin{split}
 e^{-\frac{\mu_k^2}{4} } \int_{\frac{2p\log\mu_k}{\mu_k}-\frac{\mu_k}{2}}^{\frac{\mu_k}{2}}  \frac{\mu_k}{4}e^{\tau^2} d\tau & = \frac{1}{4} + \frac{1}{2\mu_k^2}+\frac{3}{\mu_k^4} +o(\mu_k^{-4}),\\
e^{-\frac{\mu_k^2}{4} } \int_{\frac{2p\log\mu_k}{\mu_k}-\frac{\mu_k}{2}}^{\frac{\mu_k}{2}}  \frac{\tau^2 e^{\tau^2}}{\mu_k} d\tau &= \frac{1}{4}-\frac{1}{2\mu_k^2}-\frac{1}{\mu_k^4} + o(\mu_k^{-4}),
\end{split}
\]
%Integrating by parts one proves that 
%\[\begin{split}
% e^{-\frac{\mu_k^2}{4} } \int_{\frac{2p\log\mu_k}{\mu_k}-\frac{\mu_k}{2}}^{\frac{\mu_k}{2}}  \frac{\mu_k}{4}e^{\tau^2} d\tau & = \frac{1}{4} + \frac{1}{2\mu_k^2}+\frac{3}{\mu_k^4} +o(\mu_k^{-4})\\
%-e^{-\frac{\mu_k^2}{4}} \int_{\frac{2p\log\mu_k}{\mu_k}-\frac{\mu_k}{2}}^{\frac{\mu_k}{2}}
% \tau  e^{\tau^2} d\tau &=-\frac{1}{2}+o(\mu_k^{-4})\\
%e^{-\frac{\mu_k^2}{4} } \int_{\frac{2p\log\mu_k}{\mu_k}-\frac{\mu_k}{2}}^{\frac{\mu_k}{2}}  \frac{\tau^2 e^{\tau^2}}{\mu_k} d\tau &= \frac{1}{4}-\frac{1}{2\mu_k^2}-\frac{1}{\mu_k^4} + o(\mu_k^{-4}),
%\end{split}
%\]
and, summing up, we conclude
$$
(I)\le (II)= 2\pi \mu_k^{-4} +o(\mu_k^{-4}).
$$
\end{proof}

\noindent\emph{Proof of Theorem \ref{trmEnergy} (completed).} Integrating by parts and using \eqref{eqMTk} we can write
$$\|\nabla u_k\|_{L^2}^2=-\int_{B_1}u_k\Delta u_k dx=\int_{B_1}\lambda_k u_k^2 e^{u_k^2}dx.$$
Then Theorem \ref{trmEnergy} follows at once from Propositions \ref{stimabasso} and \ref{stimaalto}.
\hfill $\square$

\subsection{Some ODE theory and a crucial formula}
We conclude this section with some general lemmas analyzing the asymptotic behaviour of $w_0$ and $z_0$. In particular we will prove \eqref{derw0}, \eqref{zlog}, \eqref{intz_0}.

\begin{lemma}\label{lemmalog}
Let $f\in C^0(\R^2)$ be radially symmetric and satisfy $f(r)=O(\log^q r)$ as $r\to \infty$ for some $q\ge 0$. If $w\in C^2(\R^2)$ is a radially symmetric solution of 
\begin{equation}\label{eqgen}
-\Delta w = 4e^{2\eta_0} (f +2w) ,
\end{equation}
where $\eta_0$ is as in \eqref{eqliou}, then $\Delta w \in L^1(\R^2)$ and we have
\begin{equation}\label{ww'}
\begin{split}
w(r)&=\beta \log r +O(1)\\
w'(r)&=\frac{\beta}{r}+O\bra{\frac{\log^{\bar q} r}{r^3}},
\end{split}
\end{equation}
as $r\to \infty$, where $\bar q=\max\{1,q\}$ and 
$$\beta:=\frac{1}{2\pi}\int_{\R^2}\Delta w dx.$$
\end{lemma}

\begin{proof} We start by proving
\begin{equation}\label{wlog}
|w(r)|\le C \log{r},
\end{equation}
for some $C>0$ and $r$ sufficiently large. We consider the functions $\ph(r)= r w'(r)$  and $y(r)=(w(r), \ph(r))$. Then we can rewrite \eqref{eqgen} as
$$
y'(r)=F(r,y(r))  
$$
with 
$$
F(r,w,\ph)= \left(\frac{\ph}{r}, -4r e^{2\eta_0(r)} (f(r)+2w)\right).
$$
If we choose $R_0$ sufficiently large, so that 
$$
4 r^2 e^{2\eta_0(r)}\max\{ |f(r)|,2\} \le \frac{1}{\sqrt{2}},  \quad \text{for }r\ge R_0,
$$
then 
$$
|F(r,y)|\le \frac{1}{r}\left(1 + |y|\right) \qquad \forall \:r\ge R_0.
$$
In particular we have 
$$
|y(r)|\le |y(R_0)|+\int_{R_0}^r |F(s,y(s))| ds \le    |y(R_0)|+ \log r - \log R_0+ \int_{R_0}^r \frac{|y(s)|}{s}ds.
$$
By Gr\"onwall's lemma this yields
$$
|y(r)| \le (|y(R_0)|+\log r-\log{R_0}) \frac{r}{R_0} \le C(R_0) r \log r.
$$
In particular, 
$$
|\ph'(r)| \le r e^{2\eta_0(r)} (|f(r)|+2|w(r)|) \le C(q, R_0)  \frac{\log  r}{r^2}\in L^1((R_0,+\infty))
$$
so that $\ph (r) = r w'(r)$ is bounded and $|w(r)|\le |w(R_0)|+C \log r $ for  for $r\ge R_0$.

\medskip

Now we prove \eqref{ww'}. By the divergence theorem we have
\[
\begin{split}
2\pi rw'(r)&=\int_{B_r}\Delta w dx \\
&=2\pi \beta-\int_{\R^2\setminus B_r}\Delta wdx\\
&=2\pi \beta +O(r^{-2}\log^{\bar q} r),
\end{split}\]
where we used that, thanks to \eqref{wlog}, $-\Delta w=O(r^{-4}\log^{\bar q} r)$. This gives the second identity in \eqref{ww'}. The first one follows with the fundamental theorem of calculus.
\end{proof}

\begin{lemma}\label{lemmamagic}
Let $f$, $w$ and $\beta$ be as in Lemma \ref{lemmalog}.
Then
$$
\beta= -\frac{2}{\pi}\int_{\R^2} \frac{|x|^2-1}{(1+|x|^2)^3} f(x)dx.
$$
\end{lemma}
\begin{proof} Let us define
$$\psi(x):= \frac{|x|^2-1}{1+|x|^2},$$
which solves
$$-\Delta \psi=8e^{2\eta_0}\psi \qquad \text{in }\R^2.$$
Then for $r>0$
\[\begin{split}
4\int_{B_r} \frac{|x|^2-1}{(1+|x|^2)^3} f(x)dx &=4\int_{B_r} \psi  e^{2\eta_0}f dx \\
&=4\int_{B_r}\psi e^{2\eta_0}(f+2w)dx-8\int_{B_r}\psi e^{2\eta_0}wdx\\
&=-\int_{B_r}\psi\Delta wdx+\int_{B_r}w \Delta \psi dx=:(I).
\end{split}\]
By the divergence theorem and \eqref{ww'} we compute
\[\begin{split}
(I)&=2\pi r[\psi'(r)w(r)-\psi(r)w'(r)]\\
&=2\pi r[O(r^{-3}\log r)-(1+O(r^{-2}))r^{-1}(\beta+o(1)) ]\\
&=-2\pi\beta +o(1),
\end{split}\]
with $o(1)\to 0$ as $r\to \infty$. Letting $r\to \infty$ we conclude.
\end{proof}

We can now apply Lemma \ref{lemmalog} and Lemma \ref{lemmamagic} to the solutions $w_0$ and $z_0$ of \eqref{eqw} and \eqref{eqz}.

\begin{cor}\label{corw0}
Let $w_0$ be the solution to \eqref{eqw}. Then $w_0'$ has asymptotic behaviour \eqref{derw0}.
\end{cor}
\begin{proof}
The ODE in \eqref{eqw} corresponds to \eqref{eqgen} with
$$
f = \eta_0+\eta_0^2 =O(\log^2|x|). 
$$
Hence \eqref{derw0} follows from Lemma \ref{lemmalog} and \eqref{Deltaw0}. 
\end{proof}

\begin{cor}\label{corz0}
Let $z_0$ be the solution to \eqref{eqz}. Then $z_0$ has asymptotic behaviour \eqref{zlog}-\eqref{intz_0}.
\end{cor}
\begin{proof}
The ODE in \eqref{eqz} corresponds to \eqref{eqgen} with
$$
f = w_0 + 2w_0^2+4\eta_0 w_0 +2\eta_0^2 w_0 + \eta_0^3 + \frac{1}{2}\eta_0^4 =O(\log^4|x|). 
$$
A straightforward computation shows that
$$
\int _{\R^2} \frac{|x|^2-1}{(1+|x|^2)^3} \eta_0^3(x) dx =  -\frac{21}{4}\pi 
$$
and
$$
\int _{\R^2} \frac{|x|^2-1}{(1+|x|^2)^3}  \eta_0^4(x) dx = \frac{45}{2}\pi.
$$
Using the explicit expression \eqref{defw} of $w_0$ and integrating by parts we find
\[
\begin{split}
\int _{\R^2}  \frac{|x|^2-1}{(1+|x|^2)^3} w_0(x) dx &= \frac{\pi^3}{18}-\frac{7}{12}\pi ,\\
\int _{\R^2} \frac{|x|^2-1}{(1+|x|^2)^3} w_0(x)\eta_0(x)dx &= \bra{\frac{125}{72}-\frac{2}{3}Z(3)}\pi-\frac{2}{27}\pi^3, \\
\int _{\R^2}  \frac{|x|^2-1}{(1+|x|^2)^3} w_0(x)\eta_0^2(x) dx &= \bra{\frac{16}{9}Z(3)-\frac{409}{54}}\pi+\frac{35}{162}\pi^3+\frac{\pi^5}{45},\\
\end{split}
\]
where $Z$ denotes the Euler-Riemann zeta function. 
Finally, integrating by parts twice, we find
$$
\int _{\R^2} w_0^2  \frac{|x|^2-1}{(1+|x|^2)} dx = \bra{\frac{625}{216} -\frac{4}{9}Z(3)}\pi-\frac{1}{81}\pi^3-\frac{\pi^5}{45}.
$$
Therefore, by Lemma \ref{lemmamagic}, \eqref{ww'} holds with
$$
\beta= -\frac{2}{\pi} \int_{\R^2} \frac{|x|^2-1}{(1+|x|^2)^3} f(x) dx=- 6-\frac{\pi^2}{3}.
$$
\end{proof}

\section{Proof of Theorem \ref{trmEnergypert}}\label{sec4}
Let $u_k$ be as in the statement of the theorem, and set $r_k$ and $\eta_k$ as before in \eqref{defrk}-\eqref{defetak}.
\begin{equation}\label{eq:etakh}
\begin{split}
-\Delta \eta_k &=\lambda_k r_k^2 \mu_k e^{\mu_k^2}(1+ h(u_k)) u_k e^{2\eta_k+\frac{\eta_k^2}{\mu_k^2}}\\
&= 4e^{2\eta_0} \left(1+h\bra{\mu_k+\frac{\eta_k}{\mu_k}}\right)\left(1+\frac{\eta_k}{\mu_k^2}\right)e^{2(\eta_k-\eta_0) +\frac{\eta_k^2}{\mu_k^2}}.
\end{split}
\end{equation}
A very mild perturbation in the proof of Lemma \ref{etamu2} gives:
\begin{lemma}\label{etamu2h}
The conclusion of Lemma \ref{etamu2} still holds if we replace the ODE in \eqref{eq:etak} by \eqref{eq:etakh}, for some function $h$ with $h(t)\to 0$ as $t\to\infty$. 
\end{lemma}

Set now
$$
\delta_k:= \max\left\{\sup_{s \in [-1,1]} \left|h\left(\mu_k+\frac{s(8\log{\mu_k}+1)}{\mu_k}\right)-h(\mu_k)\right|, \frac{1}{\mu_k^6}, \frac{h(\mu_k)}{\mu_k^2}\right\}.
$$
Assuming \eqref{condh1}-\eqref{condh2}  we have $\delta_k=o(\mu_k^{-4}).$
%Fix $k$ and $S=\mu_k^4$. For $T>0$, $M>0$ to be chosen later independently on $k$. 
We also introduce the function
$$\zeta_0(r)=-1+\frac{1}{1+r^2},$$
solution to 
\begin{equation}\label{eqs1}
-\Delta \zeta_0=4e^{2\eta_0}(1+2\zeta_0).
\end{equation}
%and for some function $\phi_k$ we write
%$$\eta_k= \eta_0 + \frac{w_0}{\mu_k^2}+ \frac{z_0}{\mu_k^4} + h(\mu_k) \zeta_0 + \delta_k \phi_k.$$

\begin{lemma}\label{taylorh}
Let $0\le S\le s_k\le \mu_k^4$ and $\phi:[S,s_k]\to \R$ be given so that $\phi=o(\delta_k^{-1})$ uniformly on $[S,s_k]$.
Set
$$\eta:=\eta_0+\frac{w_0}{\mu_k^2}+\frac{z_0}{\mu_k^4}+  h(\mu_k) \zeta_0 + \delta_k \phi$$
and
\begin{equation}\label{eq:phih}
\Phi_k^h(r,\phi):=\frac{4\bra{1+h\bra{\mu_k+\frac{\eta}{\mu_k}}}\left(1+\frac{\eta}{\mu_k^2}\right)e^{2\eta+\frac{\eta^2}{\mu_k^2}} +\Delta \eta_0+\frac{\Delta w_0}{\mu_k^2}+\frac{\Delta z_0}{\mu_k^4}+h(\mu_k)\Delta \zeta_0}{\delta_k}.
\end{equation}
Then
\begin{equation}\label{stimaPhihk}
\Phi_k^h(r,\phi) =  4 e^{2\eta_0}\left(2\phi +o(1)\phi+  O(\xi^6) \right),\quad  \text{in } [S,s_k],
\end{equation}
where $\xi$ is as in \eqref{defxi}.
%$\xi(r)=1+\log(1+r)$.
\end{lemma}

\begin{proof} The proof is similar to the one of Lemma \ref{taylor}.  Using the logarithmic growth of $\eta_0$, $w_0$, $z_0$,  the bound on $s_k$, and the definition of $\delta_k$, we expand
\[\begin{split}
\psi&:= 2\eta+\frac{\eta^2}{\mu_k^2} -2\eta_0\\
&=  \frac{2w_0+\eta_0^2}{\mu_k^2}+\frac{2z_0+2\eta_0w_0}{\mu_k^4} +2h(\mu_k)\zeta_0+2\delta_k\phi +o(1)\delta_k\phi +
O(\delta_k \xi^2),\\
\psi^2&= \frac{4w_0^2 + 4w_0\eta_0^2+ \eta_0^4}{\mu_k^4} + o(1)\delta_k \phi+O(\delta_k\xi^4),\\
\psi^3&= O(\delta_k\xi^6)+o(1)\delta_k\phi.
\end{split}\]
Then $\psi$ is uniformly bounded for $r\in [S,s_k]$ and we can write
\[
e^{\psi}-1 -\psi -\frac{\psi^2}{2} = o(1)\delta_k\phi+O(\delta_k\xi^6).
\]
Therefore
\begin{equation}\label{expexph}
\begin{split}
e^{\psi}&= 1+ \frac{2w_0+\eta_0^2}{\mu_k^2} +\frac{2z_0+2w_0^2 +2\eta_0 w_0+2w_0 \eta_0^2+\frac{1}{2}\eta_0^4}{\mu_k^4}  +2h(\mu_k)\zeta_0+2\delta_k\phi \\
&\qquad + o(1)\delta_k \phi+ O(\delta_k \xi^6).
\end{split}\end{equation}
Furthermore
\begin{equation}\label{expeta}
1+\frac{\eta}{\mu_k^2}=1+\frac{\eta_0}{\mu_k^2}+ \frac{w_0}{\mu_k^4}+ o(1)\delta_k \phi+O(\delta_k\xi),
\end{equation}
and, since  $|\eta(r)|\le  8\log\mu_k+1$ for $r\in [S,\mu_k^4]$ and $k$ large,  the definition of $\delta_k$ gives 
\begin{equation}\label{exph}
1+h\bra{\mu_k+\frac{\eta}{\mu_k}}=1+h(\mu_k)+O(\delta_k).
\end{equation}
Finally, multiplying \eqref{expexph} by \eqref{expeta}-\eqref{exph} and using \eqref{eqliou}, \eqref{eqw}, \eqref{eqz} and \eqref{eqs1}, we obtain \eqref{stimaPhihk}.
\end{proof}

\begin{rmk}
Our choice of the bound  $s_k\le \mu_k^4$ is strictly connected to the regularity assumptions on $h$. If one replaces \eqref{condh2} with the simpler (but stronger) assumption
\begin{equation}\label{condhbis}
\lim_{t\to \infty} \sup_{|s|\le L} t^4|h(t+s)-h(t)| =0 \qquad \forall\: L>0,
\end{equation}
\end{rmk}
then it is possible to obtain 
$$
\Phi_k^h(r,\phi) =  4 e^{2\eta_0}\left(2\phi +o(1)\phi+ O(\mu_k^{-2}\xi^2)\phi + O(\xi^6) \right)\quad  \text{in } [S,e^{\mu_k}],
$$
precisely as in Lemma \ref{taylor}. However,  considering as a model problem $h(t) =t^{-p}$ for large $t$, \eqref{condhbis} is satisfied only for $p>3$, while the condition \eqref{condh2} allows to consider any $p>2$. 
Alternatively, the scale of the Taylor expansions can be improved by considering further terms in the expansion \eqref{exph}, see Section \ref{sec5}.

\begin{prop}\label{lemmaerrorh} There exist $M>0$ and $T>0$ such that
$$\eta_k=\eta_0+\frac{w_0}{\mu_k^2}+\frac{z_0}{\mu_k^4}+h(\mu_k)\zeta_0+\delta_k\phi_k$$
with $|\phi_k|\le M\xi$ on $[0, \mu_k^4]$ and $|\phi_k'(r)|\le \frac{M}{r}$ on $[T,\mu_k^4]$, where $\xi$ is as in \eqref{defxi}.
\end{prop}

\begin{proof} Nothing changes from the proof of Proposition \ref{lemmaerror}, since the structure and bounds of the equation
$$-\Delta \phi_k=\Phi_k^h(r,\phi_k)$$
satisfied by $\phi_k$, as given by Lemma \ref{taylorh}, are the same.
\end{proof}

\begin{prop}\label{stimabassoh} Given a sequence $(s_k)$ with $s_k\in [\mu_k^p,\mu_k^4]$ for some $p\in (2,4]$, we have
$$\int_{B_{r_ks_k}} \lambda_k (1+h(u_k)) u_k^2e^{u_k^2}dx= 4\pi+\frac{4\pi}{\mu_k^4}+o(\mu_k^{-4}).$$
\end{prop}

\begin{proof} We start writing
\[
\begin{split}
(I)&:=\int_{B_{r_k s_k}}\lambda_k(1+h(u_k)) u_k^2e^{u_k^2}dx\\
&=4\int_{B_{s_k}}\left(1+\frac{\eta_k}{\mu_k}\right)^2  \left(1+h\left(\mu_k+\frac{\eta_k}{\mu_k}\right)\right) e^{2\eta_k+\frac{\eta_k^2}{\mu_k^2}}dx\\
&= \int_{B_{s_k}}\left(1+\frac{\eta_k}{\mu_k^2}\right) \left(-\Delta \eta_0 - \frac{\Delta w_0}{\mu_k^2} - \frac{\Delta z_0}{\mu_k^4 }-h(\mu_k)\Delta \zeta_0 +\delta_k\Phi_k^h(r,\phi_k) \right)dx.
\end{split}
\]
Using Lemma \ref{taylorh} and Proposition \ref{lemmaerrorh} we have on $[0,s_k]$
$$1+\frac{\eta_k}{\mu_k^2}= 1+\frac{\eta_0}{\mu_k^2}+\frac{w_0}{\mu_k^4}+O(\delta_k), $$
and
$$\Phi_k^h(r, \phi_k)=O(e^{2\eta_0}\xi^6). $$
Arguing as in Proposition \ref{stimabasso} we get
\[\begin{split}
(I)&=\int_{B_{s_k}}\left(-\Delta\eta_0-\frac{\eta_0\Delta\eta_0+\Delta w_0}{\mu_k^2}-\frac{w_0\Delta\eta_0+\eta_0\Delta w_0+\Delta z_0}{\mu_k^4} -h(\mu_k)\Delta \zeta_0\right)dx + O(\delta_k)\\
&=:(I_0)+\frac{(I_2)}{\mu_k^2}+\frac{(I_4)}{\mu_k^4}-h(\mu_k)\int_{B_{s_k}}\Delta \zeta_0dx+O(\delta_k).
\end{split}\]
As before we have
\[\begin{split}
(I_0)&=4\pi+o(\mu_k^{-4})\\
(I_2)&= O(s_k^{-2}\log^2 s_k)=o(\mu_k^{-2})\\
(I_4)&=4\pi+o(1).
\end{split}\]
Finally,
$$h(\mu_k)\int_{B_{s_k}}\Delta \zeta_0dx=2\pi h(\mu_k) r \zeta_0'(s_k) =h(\mu_k)O(s_k^{-3})=o(\mu_k^{-4}),$$
and we conclude.
\end{proof}

\noindent\emph{Proof of Theorem \ref{trmEnergypert} (completed).} Again integrating by parts we infer for some $p\in (2,4]$
\[\begin{split}
\|\nabla u_k\|_{L^2}^2&=-\int_{B_1}u_k\Delta u_k dx\\
&=\int_{B_{\mu_k^p r_k}}\lambda_k (1+h(u_k))u_k^2 e^{u_k^2}dx+\int_{B_1\setminus B_{\mu_k^p r_k}}\lambda_k (1+h(u_k))u_k^2 e^{u_k^2}dx\\
&=:(I)+(II).
\end{split}\]
The term $(I)$ is bounded from above and below by Proposition \ref{stimabassoh}. For the term $(II)$ we use that  $\eta_k(r)\le \eta_0(r)$ for $r\ge \mu_k^p$ and $k$ large enough, which follows from Lemma \ref{etaketa0}. Then the proof of Proposition \ref{stimaalto} can still be applied and we infer
$$(II)\le \frac{2\pi(1+\sup h)}{\mu_k^4} +o(\mu_k^{-4}).$$
Summing up $(I)$ and $(II)$ we conclude.
\hfill$\square$

\section{Proof of Theorem \ref{trmexample}} \label{sec5}
Since the perturbation $h(t)$ is now of order $t^{-2}$, its presence will change the Taylor expansion of the right-hand side of
\eqref{eq:etakh} already at order $\mu_k^{-2}.$
As a consequence we will see that the function $\mu_k^2(\eta_k-\eta_0)$ will converge to a new function $w_a$, solution to 
\begin{equation}\label{eqwa}
\left\{
\begin{array}{l}
-\Delta w_a= 4e^{2\eta_0}(\eta_0+\eta_0^2-a+2w_a)\text{ in }\R^2\\
w_a(0)=w_a'(0)=0.\rule{0cm}{0.5cm}
\end{array}
\right.
\end{equation}
Since
$$-\Delta (w_a-w_0)= 4e^{2\eta_0}(-a+2(w_a-w_0)),$$
we have $w_a-w_0=-a \zeta_0$.

Also the function $z_0$ will be replaced by $z_a$ which satisfies 
\[
\left\{
\begin{array}{l}
-\Delta z_a= 4e^{2\eta_0}( a (\eta_0 -\eta_0^2-2w_a )+w_a+2w_a^2+4\eta_0w_a +2\eta_0^2 w_a+\eta_0^3 +\frac12 \eta_0^4+2z_a) \text{ in }\R^2\rule{0cm}{0.5cm}\\
z_a(0)=z_a'(0)=0,\rule{0cm}{0.5cm}
\end{array}
\right.
\]
and differs from $z_0$ by the solution to 
\[
\left\{
\begin{array}{l}
-\Delta (z_a-z_0)= 4e^{2\eta_0}[2 a^2(\zeta_0+\zeta_0^2)+a(\eta_0-\eta_0^2-2w_0+\zeta_0(-2\eta_0^2-4\eta_0-4w_0-1)\\
\rule{5cm}{0cm}+  2(z_a-z_0)]\text{ in }\R^2\rule{0cm}{0.5cm}\\
z_a(0)-z_0(0)=z_a'(0)-z_0'(0)=0.\rule{0cm}{0.5cm}
\end{array}
\right.
\]

Then with Lemma \ref{lemmamagic} we have
\begin{equation}\label{zabeta}
z_a(r)-z_0(r)=\beta\log r+O(1),\quad \frac{1}{2\pi}\int_{\R^2}\Delta (z_a-z_0)dx=\beta,
\end{equation}
with $\beta=\beta_1+\beta_2$, where for $\psi_0(x):=\frac{|x|^2-1}{(1+|x|^2)^3}$,
\[\begin{split}
\beta_1&=-\frac{2a}{\pi}\int_{\R^2}(\eta_0-\eta_0^2-2w_0+\zeta_0(-2\eta_0^2-4\eta_0-4w_0-1))\psi_0 dx\\
\beta_2&=-\frac{2a^2}{\pi}\int_{\R^2}2( \zeta_0+\zeta_0^2)\psi_0 dx.
\end{split}\]
One can compute
\[\begin{array}{ll}
\displaystyle\frac{2}{\pi}\int_{\R^2}\zeta_0^2\psi_0 dx=\frac13,& \quad \displaystyle\frac{2}{\pi}\int_{\R^2}(-2w_0)\psi_0 dx=\frac{7}{3}-\frac{2\pi^2}{9},\\
\displaystyle\frac{2}{\pi}\int_{\R^2}\eta_0\psi_0 dx=-1,&\quad \displaystyle\frac{2}{\pi}\int_{\R^2}(-\eta_0^2)\psi_0 dx=-3,\rule{0cm}{0.8cm}\\
\displaystyle\frac{2}{\pi}\int_{\R^2}(-\zeta_0)\psi_0 dx=\frac{1}{3},&\quad \displaystyle\frac{2}{\pi}\int_{\R^2}(-4w_0 \zeta_0)\psi_0 dx=-\frac{67}{27}+\frac{2\pi^2}{9},\rule{0cm}{0.8cm}\\
\displaystyle\frac{2}{\pi}\int_{\R^2}(-4\eta_0 \zeta_0) \psi_0 dx=-\frac{34}{9},&\quad \displaystyle\frac{2}{\pi}\int_{\R^2}(-2\eta_0^2 \zeta_0)\psi_0 dx=\frac{151}{27},\rule{0cm}{0.8cm}\\
\end{array}\]
hence
\begin{equation}\label{betaa}
\beta_2=0,\quad \beta=\beta_1=2a.
\end{equation}

Similar to Lemma \ref{taylor} we get

\begin{lemma}\label{taylorha}
Let $0\le S\le s_k\le e^{\mu_k}$ and $\phi:[S,s_k]\to \R$ be given so that $\phi=o(\mu_k^6)$ uniformly on $[S,s_k]$.
Set
$$\eta:=\eta_0+\frac{w_a}{\mu_k^2}+\frac{z_a}{\mu_k^4}+\frac{\phi}{\mu_k^6}$$
and (using that $h(t)=-at^{-2}$ for $t$ large)
\begin{equation}\label{eq:phiha}
\Phi_k^a(r, \phi):=\mu_k^6\left[4\bra{1-\frac{a}{\mu_k^2}\bra{1+\frac{\eta}{\mu_k^2}}^{-2} }\left(1+\frac{\eta}{\mu_k^2}\right)e^{2\eta+\frac{\eta^2}{\mu_k^2}} +\Delta \eta_0+\frac{\Delta w_a}{\mu_k^2}+\frac{\Delta z_a}{\mu_k^4}\right].
\end{equation}
Then as $k\to\infty$
$$
\Phi_k^a(r,\phi) =  4 e^{2\eta_0}\left(2\phi +o(1)\phi+ O(\mu_k^{-2}\xi^2)\phi+ O(\xi^6) \right),\quad  r\in [S,s_k],
$$
where $\xi$ is as in \eqref{defxi}.
\end{lemma}

\begin{proof} 
The proof is identical to the one of Lemma \ref{taylor}, just replacing $w_0$ and $z_0$ with $w_a$ and $z_a$ respectively, and noticing that after the Taylor expansion of the exponential in \eqref{eq:phiha} we have to consider
$$h\bra{\mu_k+\frac{\eta}{\mu_k}}=-\frac{a}{\mu_k^2}\bra{1+\frac{\eta}{\mu_k^2}}^{-2}=-\frac{a}{\mu_k^2}+\frac{2a\eta_0}{\mu_k^4} +O(\mu_k^{-6}\xi^2),$$
as $k\to \infty$.
\end{proof}

With the same proof of Proposition \ref{lemmaerror} (using Lemma \ref{taylorha} instead of Lemma \ref{taylor}) we get:

\begin{prop}\label{lemmaerrorha} There exist $M>0$ and $T>0$ such that
$$\eta_k=\eta_0+\frac{w_a}{\mu_k^2}+\frac{z_a}{\mu_k^4}+\frac{\phi_k}{\mu_k^6}$$
with $\phi_k$ satisfying \eqref{stimaphik} for $k$ large.
\end{prop}

\begin{prop}\label{stimabassoha} Given a sequence $(s_k)$ with $s_k\in [\mu_k^p,e^{\mu_k}]$ for some $p>2$, and $h(t)=-at^{-2}$ for $t$ large, we have
$$\int_{B_{r_ks_k}} \lambda_k (1+h(u_k)) u_k^2e^{u_k^2}dx= 4\pi+\frac{4\pi-4\pi a}{\mu_k^4}+o(\mu_k^{-4}).$$
\end{prop}

\begin{proof} Write as in the proof of Theorem \ref{trmEnergy}
\[
\begin{split}
(I)&:=\int_{B_{r_k s_k}}\lambda_k (1+h(u_k))u_k^2e^{u_k^2}dx\\
&= \int_{B_{s_k}}\left(1+\frac{\eta_k}{\mu_k^2}\right) \left(-\Delta \eta_0 - \frac{\Delta w_a}{\mu_k^2} - \frac{\Delta z_a}{\mu_k^4 } +\frac{\Phi_k^a(r,\phi_k)}{\mu_k^6} \right)dx,
\end{split}
\]
where $\Phi_k^a$ is as in \eqref{eq:phiha}. Proceeding as in the Proof of Theorem \ref{trmEnergy} and using Lemma \ref{taylorha} and Proposition  \ref{lemmaerrorha} we have
\[\begin{split}
(I)&=\int_{B_{s_k}}\left(-\Delta\eta_0-\frac{\eta_0\Delta\eta_0+\Delta w_a}{\mu_k^2}-\frac{w_a\Delta\eta_0+\eta_0\Delta w_a+\Delta z_a}{\mu_k^4} \right)dx + O(\mu_k^{-5})\\
&=:(I_0)+\frac{(I_2^a)}{\mu_k^2}+\frac{(I_4^a)}{\mu_k^4}+O(\mu_k^{-5}).
\end{split}\]
As before we have $(I_0)=4\pi+o(\mu_k^{-4})$, while for $(I_2^a)$ replacing $w_0$ with $w_a$ leads us to the extra term
$$-a \int_{B_{s_k}} \Delta \zeta_0 dx=-2\pi s_k a \zeta_0'(s_k)=O(s_k^{-2})=o(\mu_k^{-2}).$$
Therefore we have again $(I_2^a)=(I_2)+o(\mu_k^{-2})= o(\mu_k^{-2})$.

As for the remaining term we have
\[\begin{split}
(I_4^a)&=(I_4)+a\int_{B_{s_k}}(\zeta_0\Delta \eta_0+\eta_0\Delta\zeta_0)dx-\int_{B_{s_k}}\Delta (z_a-z_0)\\
&=4\pi + ao(1)- 2\pi\beta+o(1)\\
&= 4\pi -4\pi a+o(1),
\end{split}\]
where we used \eqref{zabeta} and \eqref{betaa}. Summing up we conclude.
\end{proof}

\noindent\emph{Proof of Theorem \ref{trmexample} (completed).} As in the proof of Theorems \ref{trmEnergy} and \ref{trmEnergypert}, it suffices to add the estimate of Proposition \ref{stimabassoha} to the estimate of Proposition \ref{stimaalto}. For the latter we use Lemma \ref{etaketa0}.
\hfill$\square$

\section{Proof of Theorem \ref{punticritici}}\label{sec6}

Using ODE theory as in the proof of Theorem 1 in \cite{MM}, for every $\mu>0$ one can find exactly one $\lambda_\mu>0$ and one solution $u_\mu$ to the problem
\begin{equation}\label{eqELvin}
\left\{
\begin{array}{ll}
-\Delta u_\mu= \lambda_\mu (1+h(u_\mu))u_\mu e^{u_\mu^2}& \text{in }B_1\\
u_\mu=0&\text{on }\de B_1\\
u_\mu >0&\text{in }B_1\\
u_\mu(0)=\mu,& 
\end{array}
\right.
\end{equation}
where we used that the nonlinearity
$$t\mapsto (1+h(t))te^{t^2}=\bra{t+g(t)t+\frac{g'(t)}{2}}e^{t^2}$$
is positive for $t>0$ and of class $C^1$.
Moreover the function $\M{E}(\mu):=\|u_\mu\|^2_{H^1_0}$ is continuous and
$$\lim_{\mu\to 0}\M{E}(\mu)=0.$$
Thanks to Theorem \ref{trmEnergypert} we have that
$$\lim_{\mu\to\infty} \M{E}(\mu)=4\pi$$
and $\M{E}(\mu)>4\pi$ for large values of $\mu$.
In particular
$$\Lambda^*:=\sup_{\mu> 0}\M{E}(\mu)>4\pi.$$
By continuity for every $\Lambda\in (4\pi,\Lambda^*)$ there exists at least $2$ solutions $u_{\mu_1}$ and $u_{\mu_2}$ with $\M{E}(u_{\mu_1})=\M{E}(u_{\mu_2})=\Lambda$. Then $u_{\mu_1}$ and $u_{\mu_2}$ are critical points of $E|_{M_\Lambda}$, since \eqref{eqELvin} is the Euler-Lagrange equations of $E$ constrained to $M_{\Lambda}$.
\hfill$\square$

\section{A few more open problems}

\begin{OP}\label{OPn} Can one extend Theorem \ref{trmEnergy} (and its perturbed versions) to critical points of
$$\sup_{u\in H^{1,n}_0(B_1):\|\nabla u\|_{L^n}^n\le n^{n-1}\omega_{n-1}}\int_{\Omega}e^{u^\frac{n}{n-1}}dx\le C_n$$
in dimension $n>2$? (Here $\omega_{n-1}$ is the volume of $S^{n-1}$.)
\end{OP}

In this direction, there are some results of Adimurthi \cite{adi} and Adimurthi-Yang \cite{AY} on the solutions to
\begin{equation}\label{eqMTn}
-\Delta_n u= \lambda u|u|^{n-2} e^{u^{n'}}\quad\text{in }\Omega\Subset\R^n,
\end{equation}
from which one is led to conjecture that
$$\int_{B_1}|\nabla u_k|^n dx = n^{n-1}\omega_{n-1}+o(1),\quad \text{as }k\to\infty,$$
where $(u_k)$ is a blowing-up sequence of positive radial solution to \eqref{eqMTn}, and it would be interesting to understand the sign of the error term $o(1)$.

\begin{OP}\label{OPhigh}
Can one extend Theorem \ref{trmEnergy} to the higher-order problem
\begin{equation}\label{eqAdams}
(-\Delta)^m u=\lambda ue^{m u^2},\quad u\in H^m_0(B_1),\quad B_1\subset\R^{2m},
\end{equation}
particularly to study the existence of extremals of the Adams inequality (see \cite{ada}) on a ball?
\end{OP}
In this direction, the works \cite{mar, MS, RS, struwe} suggest that for a blowing up sequence of solutions to \eqref{eqAdams}
$$\int_{B_1}|\nabla^m u_k|^2 dx = \Lambda_1+o(1),\quad \text{as }k\to\infty,$$
where $\Lambda_1=(m-1)!\omega_{2m}$ is the total $Q$-curvature of $S^{2m}$.

\medskip

Similarly one could consider the case $m=\frac{n}{2}$ with $n$ odd, which has the additional difficult of $(-\Delta)^\frac n2$ being non-local. In this direction see \cite{IMM} and \cite{MMS}. 

\medskip

Finally we do not know what happens when we drop the assumptions \eqref{condh1}-\eqref{condh2}.

\begin{OP} Is it possible to find functions $g$ and $h$ as in \eqref{condg}-\eqref{defh} such that the energy expansion of the solutions to \eqref{eqpert0} is of the form
$$\|\nabla u_\mu\|_{L^2}^2=4\pi +\frac{A}{\mu_k^p} +o(\mu_k^{-p}),\quad \mu_k=u_k(0)\to\infty$$
for some $A\ne 0$, $p<4$? Can one even take $p\le 2$?
\end{OP}

\end{document}